\documentclass[11pt,a4paper,reqno]{amsart}  
\usepackage[utf8]{inputenc}
\usepackage{graphicx} 
\usepackage[left=1.75cm,right=1.75cm]{geometry}

\title{Almost sparse operators and Bergman projectors on trees}

\author[Conde-Alonso]{Jos\'{e} M. Conde Alonso}
\address{Jos\'{e} M. Conde Alonso \hfill\break\indent 
 Departamento de Matem\'aticas \hfill\break\indent 
 Universidad Aut\'onoma de Madrid \hfill\break\indent 
 C/ Francisco Tom\'as y Valiente sn\hfill\break\indent 
 28049 Madrid, Spain}
\email{jose.conde@uam.es}

\author[Monti]{Matteo Monti}
\address{Matteo Monti \hfill\break\indent 
Dipartimento di Matematica e Applicazioni \hfill\break\indent
Università degli studi di Milano-Bicocca \hfill\break\indent
Via Roberto Cozzi 55\hfill\break\indent
20125 Milano, Italy}
\email{matteo.monti@unimib.it}

\author[Rizzo]{Elena Rizzo}
\address{Elena Rizzo\hfill\break\indent
Departamento de An\'alisis Matem\'atico\hfill\break\indent
Facultad de Ciencias\hfill\break\indent
Universidad de M\'alaga\hfill\break\indent
Bulevar Louis Pasteur, s/n, Campus de Teatinos\hfill\break\indent
29071 Málaga, Spain}
\email{elena.rizzo@uma.es}

\usepackage[utf8]{inputenc}

\usepackage{amssymb,amsmath,amsopn,amsfonts,pb-diagram,accents,xcolor,array,hyperref,mathtools,enumitem,bm,bbm,mathrsfs,amscd,comment,algpseudocode,multicol,float,tikz}

\usepackage[subnum]{cases}

\newtheorem{theorem}{Theorem}[section]
\newtheorem{lemma}[theorem]{Lemma}
\newtheorem{definition}[theorem]{Definition}
\newtheorem{example}[theorem]{Example}
\newtheorem{lem}[theorem]{Lemma}
\newtheorem{prop}[theorem]{Proposition}

\newtheorem{rem}[theorem]{Remark}

\numberwithin{equation}{section}

\newcommand{\N}{\mathbb{N}}

\newcommand{\R}{\mathbb{R}}
\newcommand{\C}{\mathbb{C}}

\newcommand{\avg}[2]{\langle{#1}\rangle_{#2}}

\newcommand{\X}{\mathfrak{X}}
\newcommand{\G}{\mathfrak{G}}
\newcommand{\T}{\mathcal{T}}
\newcommand{\U}{\mathcal{U}}
\newcommand{\Qp}{\dot{Q}}
\newcommand{\Q}{\mathcal{Q}}

\newcommand{\anc}[2]{{#1}^{(#2)}}

\newcommand{\Ss}{\mathcal{S}}

\newcommand{\As}{\mathcal{A}_{\Ss}} 

\newcommand{\Proj}{\mathcal{P}}

\newcommand{\BBv}{\widetilde{{B}_p}}

\allowdisplaybreaks

\begin{document}


\allowdisplaybreaks

\begin{abstract}
We develop a sparse domination framework adapted to nested families in
general measure spaces, possibly including point masses. We show that the Carleson condition is equivalent to a slight generalization of the usual sparse condition, and use this characterization to obtain one- and two-weight estimates for the associated averaging operators, together with converse testing conditions. We apply our general results to integral operators on locally finite trees. Trees come naturally equipped with two nested geometries: sectors and pruned
sectors. We prove pointwise sparse domination of natural versions of the harmonic Bergman projector by the corresponding averaging operators, depending on whether the measure on the tree is doubling or not, and show that the distinction is necessary in general. For the positive Bergman projector, we obtain pointwise comparability. As a consequence, we obtain new weighted inequalities and boundedness results for the Bergman projector and a class of Toeplitz-like multipliers. 
\end{abstract}

\maketitle

\section*{Introduction}

Let $(\G,\mu)$ be a diffuse measure space, and $\Ss =\{S_j\}_j$ be a family of sets. The Carleson packing condition
$$
\sum_{S \in \Ss: S \subseteq \Omega} \mu(S) \leq \Lambda \mu\left(\Omega\right), \qquad\text{for all}\qquad \Omega \subseteq \G,
$$
is equivalent to $\Ss$ being $1/\Lambda$-sparse: for each $S \in \Ss$, there exists $E_S \subseteq S$ such that $\mu(E_S) \geq 1/\Lambda \mu(S)$ and such that the family $\{E_S\}_{S\in\Ss}$ is disjoint. This was shown in \cite{Verb96,LN2019} for dyadic cubes on $\mathbb{R}^n$ and in general in \cite{Hann18} (see also \cite{Lorist,Rey}). Sparse families of sets have found widespread use in harmonic analysis after Lerner's seminal contributions \cite{Lerner13,Lerner13-2}, which proposed to dominate -- in a suitable sense -- complicated operators $T$ by the averaging operators
$$
\As f(x) = \sum_{S\in \Ss} \langle f \rangle_S 1_S(x),
$$
whose quantitative behavior is much easier to understand. Above and in what follows, the integral averages
$$
\avg{f}{S}=\frac{1}{\mu(S)} \int_S f \; d\mu
$$
are always computed with respect to the measure $\mu$, that we omit from the notation, and we only consider sets $S$ with $0<\mu(S)<\infty$. A reasonably complete account of the trend that is now called sparse domination can be found in \cite{Pereyra,Duarte} and the references therein.

The equivalence between the Carleson packing condition and the sparsity one is useful because the former is easier to check in practice, while the latter is easier to use when working with $\As$. The hypothesis of $(\G,\mu)$ being diffuse is necessary for it to hold. In the presence of point masses, as in discrete contexts, a suitable replacement needs to be found. One option is to allow the existence of the sets $E_S$ to be replaced by that of functions $0 \leq f_S \leq 1$ whose sum is bounded by $1$. Another, simpler one, which will be what we choose in the sequel, is to simply allow the sets $E_S$ some bounded overlap. It turns out that the resulting notion of almost sparsity leads to a theory of operators $\As$ with similar properties to the classical ones. To that end, given a nested family $\Ss$, define
$$
\mathcal{C}_\Ss := \Bigl\{\nu \mbox{ measure on } \G: \mbox{ there exists } \Lambda \geq 1 \mbox{ such that } \Ss \mbox{ is } \Lambda\mbox{-Carleson with respect to }\nu\Bigr\}.
$$
Then the operator $\As$, defined with respect to any $\nu \in \mathcal{C}_\Ss$, satisfies the same $L^p$-boundedness results, weak-type bounds, and weighted inequalities -- with one or two weights -- as the classical sparse operators with respect to dyadic cubes and with the Lebesgue measure. Moreover, given any $\nu \in \mathcal{C}_\Ss$, the natural $A_\infty$-type class of weights for $\As$ has a simple geometric interpretation: it is just the class of weights $w$ such that $w\; d\nu \in \mathcal{C}_\Ss$.

We next explore applications of almost sparse operators in the discrete setting. The underlying measure space is a locally finite tree $\X$, hooked to a point at infinity. Such $\X$ comes equipped with two useful, distinct families of nested sets. A sector $T_x$ contains a vertex $x$ and everything below it; the pruned sector $U_x=T_x\setminus\{x\}$ contains the descendants but not
the vertex itself. Both the family of sectors $\T$ and the family of pruned sectors $\U$ are nested. The first one corresponds to the family of balls with respect to the Gromov metric on the tree, while the second corresponds to the discrete analogue of the bottom half of the family of Carleson boxes on the upper half-plane. We also work with their union $\Q=\T\cup\U$, which is again a nested family. We have
$$
\mathcal{C}_\Q=\mathcal{C}_\T=\mathcal{C}_\U,
$$
but for each fixed $\nu \in \mathcal{C}_\Q$, it may very well be that the sets in $\T$ have very different measure to those in $\U$ which are very close to them. In particular, $T_x$ and $U_x$ may have noncomparable measure. In our first example, this happens if the number of successors of each vertex $x$ is not a bounded quantity over $\X$. In that case, the measure space $(\X,\nu)$ is not dyadically doubling, and that is reflected in the results that we obtain. The distinction between both cases appears naturally in kernel estimates. For two vertices $x,y$, their paths towards the fixed end first meet at their confluent $x\wedge y$.  If one vertex lies in the sector of the other, the relevant scale is $\nu(T_{x\wedge y})$; if they lie in different branches below their
confluent, it is $\nu(U_{x\wedge y})$. Indeed, if we write the almost sparse operator
$$
\mathcal{A}_\Q f(x) = \sum_{Q \in \Q} \avg{f}{Q} 1_Q(x) = \sum_{y\in\X} K_{\mathcal{A}_\Q}(x,y) f(y) \nu(y),
$$
we find that 
\[
K_{\mathcal{A}_\Q}(x,y) \sim
\begin{dcases*}
    \nu(T_{x\wedge y})^{-1}, & $x\wedge y \in \{x,y\}$,
    \\\nu(U_{x\wedge y})^{-1},& \text{otherwise}.
\end{dcases*}
\]

In the body of the paper, we show that the kernel $K_\mathcal{P}$ of the harmonic Bergman projector on a radial tree equipped with the natural measure $\mu_\alpha$, $\alpha>1$, studied in
\cite{CADMMRV2025,DMMV2023}, satisfies $|K_{\mathcal{P}}| \lesssim K_{\mathcal{A}_\Q}$, and so we recover a simpler version of the main sparse domination result in \cite{CADMMRV2025}. Additionally, we obtain a pointwise equivalence theorem for the appropriate positive Bergman projector. Pointwise dyadic domination of classical Bergman-type operators appears in \cite{APR2017,RTW2017}; see also \cite{Brocchi2025} for refined two-weight estimates. Radiality, a condition that seems to be close to optimal, makes the harmonic basis and its kernel accessible, but it does not force doubling when the number of successors per vertex is unbounded. It turns out that $K_{\mathcal{A}_\Q} \sim K_{\mathcal{A}_\T}$ exactly if the opposite happens, illustrating the need for a different sparse domination principle in the nondoubling setting. Our formulation not only simplifies the sparse domination result from \cite{CADMMRV2025}, that we can now attain with a strikingly simple argument, but it improves it: the class of weights for which we now find $L^p$-estimates is larger, and optimal. Moreover, we obtain two-weight estimates, which allows us to consider certain estimates for harmonic Toeplitz multipliers. 

We also consider the forward harmonic Bergman projector associated with a flow measure, introduced in \cite{OS2026}. Here harmonicity averages only over successors, rather than over all neighbours, and so the resulting operator is of order 1. We obtain sparse domination for $\mathcal{P}$ and equivalence for the positive version in terms of $\mathcal A_\Q$. This is entirely new in this setting and entails the usual weighted corollaries as well. 

We end the paper with a discussion of the weight classes arising in our weighted inequalities. These are defined as intersections of the classical Muckenhoupt classes associated with $\T$ and $\U$, and we investigate whether any inclusion between the two underlying classes can hold. As suggested by the necessity of introducing $\U$ in the unbounded setting, no such containment occurs in general when the branching is unbounded. Thus, also at the level of weighted theory, the two geometries encoded by $\T$ and $\U$ remain genuinely distinct. The rest of the paper is organized as follows: Section \ref{sec:sec1} contains the general theory of almost sparse operators and a general sparse comparison principle. The applications to trees are described in Section \ref{sec:sec2}. 


\section{Almost sparse operators} \label{sec:sec1}

\noindent\textbf{Almost sparse families.} Fix a measure $\nu$ on $\G$, which may have atoms. We always work with countable families $\Ss$ of measurable subsets of $\G$ such that $0<\nu(S)<\infty$ for each $S\in\Ss$, and which are nested: given $S,P \in \Ss$, then $S\cap P\in \{\emptyset,S,P\}$. Such an $\Ss$ is said to be almost sparse if for every  $0<\eta<1$ there exists $N\in \mathbb N$ such that for every $S\in\Ss$, there exists $E_S\subseteq S$ such that $\nu(E_S)\geq \eta\nu(S)$, and the sets $E_S$ satisfy
\begin{equation} \label{eq: alm sp}
\sum_{S\in\Ss}1_{E_S}\leq N. 
\end{equation}

\begin{lem}\label{lem: equiv Carl}
    $\nu \in \mathcal C_{\Ss}$ if and only if $\Ss$ is almost sparse with respect to $\nu$.
\end{lem}

\begin{proof}
    The almost sparsity of $\Ss$ clearly implies $\nu \in \mathcal{C}_\Ss$. Indeed,
    \[ \sum_{\substack{P\in\Ss\\ P\subseteq S}}\nu(P)\leq \frac 1\eta \sum_{\substack{P\in\Ss\\ P\subseteq S}}\nu(E_P)=\frac 1\eta \int_S\sum_{\substack{P\in\Ss\\ P\subseteq S}}1_{E_P}(x) \; d\nu(x)\leq \frac N\eta \nu(S).    \]
    To prove the converse, choose a Carleson constant $\Lambda$, fix $0<\eta<1$, and let $N=\lfloor\Lambda(1-\eta)^{-1}\rfloor$. For every $S\in \Ss$, we define 
    \[ \mathcal{E}_S(x):=\{P\in \Ss\colon x\in P\subseteq S\},\qquad E_S:=\{x\in S\colon \#\mathcal{E}_S(x)\leq N\},\]
    with the convention $\mathcal E_S(x)=\emptyset$ when $x\not \in S$. We have
    \begin{align*}
        \Lambda\nu(S)\geq \sum_{\substack{P\in\Ss\\ P\subseteq S}}\nu(P)= \int_{S} \#\mathcal{E}_S(x)\; d\nu (x) \geq \int_{S \setminus E_S} 
       \#\mathcal{E}_S(x)\; d\nu(x)\geq (N+1)\nu(S \setminus E_S), 
    \end{align*}
Hence we have that
     \[\nu(E_S)=\nu(S)-\nu(S \setminus E_S)\geq \Big(1-\frac \Lambda {N+1}\Big)\nu(S)\geq\eta\nu(S).\]
    Furthermore, for every $x\in\X$ it is easy to see that $\#\{S\in \Ss\colon x\in S,\#\mathcal{E}_S(x)\leq N\}\leq N$. Indeed, if it were greater than $N$, then there would exist $S_1\subsetneq\dots\subsetneq S_{N+1}\in\Ss$ distinct such that $\#\mathcal{E}_{S_j}(x)\leq N$ and $x\in S_j$ for every $j=1,\dots,N+1$. Thus we would have that $\#\mathcal{E}_{S_{N+1}}(x)\geq N+1$, a contradiction.
    Therefore,
    $$
        \sum_{S\in \Ss}1_{E_S}(x)=\#\{S\in \Ss\colon x\in E_S\}=\#\{S\in\Ss\colon x\in S,\mathcal{E}_S(x)\leq N\}\leq N.
    $$
\end{proof}

We note that the definition of the sets $E_S$ depends only on the geometry and is independent of $\nu$. Also, by construction, the overlap number $N$ must be larger than $\Lambda-1$, where $\Lambda$ is the Carleson constant. On the other hand, if $\Lambda<2$ then for certain choices of $\eta$, namely $\eta<1-\Lambda/2$, we can choose $N=1$. In that case $\{E_S\}_{S\in\Ss}$ is disjoint and $\Ss$ is $\eta$-sparse.

\begin{prop}\label{prop: somma sect}
    Let $\nu\in \mathcal C_{\Ss}$ with constant at most $\Lambda$. For every $S\in\Ss$ we have
    \[
\sum_{\substack{P\in\Ss\\P\supseteq S}} \frac{1}{\nu(P)} \leq \frac{\Lambda^2}{\nu(S)}.
    \]
\end{prop}

\begin{proof}
    Since $\Ss$ is nested, we have that the family of $P\in\Ss$ that contain a fixed $S\in\Ss$ is an ordered nested sequence $S=P_0\subsetneq P_1\subsetneq\dots\subsetneq P_K\subsetneq\dots$. For simplicity, we set 
    \[\Sigma_K=\sum_{\ell=0}^K\nu(P_\ell),\qquad K\in\mathbb N.\]
    For $K\in\mathbb N$, since $\nu \in\mathcal C_{\Ss}$, we have that $\Sigma_K\leq \Lambda(\Sigma_K-\Sigma_{K-1})$. Hence we get $\Sigma_{K-1}\leq (1-\Lambda^{-1})\Sigma_K$ and iterating
    one has
    \[\nu(P_K)=\Sigma_K-\Sigma_{K-1}\geq \frac 1{\Lambda}\Sigma_K\geq \frac 1{\Lambda}\Big(1-\frac 1{\Lambda}\Big)^{-K}\nu(S). \]
    Finally we conclude by
    \[\sum_{\ell=0}^{\infty}\frac 1{\nu(P_\ell)}\leq \frac \Lambda{\nu(S)}\sum_{\ell=0}^{\infty}\Big(1-\frac 1{\Lambda}\Big)^\ell=\frac {\Lambda^2}{\nu(S)}.\]
\end{proof}

\noindent\textbf{Weighted estimates for averaging operators.} Given $\Ss$ and $\nu \in \mathcal{C}_\Ss$, the associated positive almost sparse operator is 
\[
\mathcal A_\Ss f(x)=\sum_{S\in\Ss}\avg {|f|}S 1_S(x),\qquad x\in\G.
\]

Since $\Ss$ is nested, the maximal operator
$$
\mathcal{M}_\Ss f(x) = \sup_{x\in S\in  \Ss} \avg{|f|}{S},
$$
with averages defined with respect to any measure $\mu$, is bounded on $L^p(\mu)$ for all $1<p<\infty$ and is of weak type $(1,1)$, with constants independent of $\mu$. We consider the standard Muckenhoupt classes of weights defined in terms of $\Ss$:

\begin{align*}
    A^{\Ss}_p(\nu)&=\Big\{\omega\colon \G \to (0,+\infty)\colon \omega>0,[\omega]_{p,\Ss}:=\sup_{S\in\Ss}\avg{\omega}{S}\avg{\omega^{-\frac 1{p-1}}}{S}^{p-1}<+\infty\Big\},\\
    A^{\Ss}_1(\nu)&=\Big\{\omega\colon \G \to (0,+\infty)\colon \omega>0,[\omega]_{1,\Ss}:=\sup_{S\in\Ss}\avg{\omega}{S}\|\omega^{-1}1_S\|_\infty<+\infty\Big\},\\
    A^{\Ss}_\infty(\nu)&=\Big\{\omega\colon \G \to (0,+\infty)\colon \omega>0,[\omega]_{\infty,\Ss}:=\sup_{S\in\Ss}\frac 1{\omega(S)}\int_{S}\mathcal{M}_{\Ss} (\omega 1_S)(x) d\nu(x)<+\infty\Big\}.
\end{align*}
We use the notation $A_p^{\Ss}=A_p^{\Ss}(\nu)$ when $\nu$ is clear from the context. The $L^p$-boundedness of $\mathcal{M}_\Ss$, $1<p<\infty$, can be used to prove that $\As$ maps $L^p(w)$ to itself whenever $w\in A^{\Ss}_p$. Regarding $A^{\Ss}_\infty(\nu)$, an adaptation of the Carleson embedding theorem yields the geometric characterization claimed in the Introduction.

\begin{lemma}\label{lem:CET}
    Let $\nu\in\mathcal C_{\Ss}$ and $\Lambda>1$ be such that $\Ss$ is $\Lambda$-Carleson in $(\G,\nu)$. For any measurable $f\geq 0$ and $S\in\Ss$, we have
    \[\sum_{\substack{P\in\Ss\\P\subseteq S}}\int_P f(x)d\nu(x)\leq \Lambda\int_S \mathcal M_\Ss(f1_S)(x)d\nu(x).\]
\end{lemma}
\begin{proof}
    Let $f\geq0$ and $S\in\Ss$. By the layer-cake formula,
\[
\sum_{\substack{P\in\Ss\\P\subseteq S}}\int_P f(x)d\nu(x)
=
\sum_{\substack{P\in\Ss\\P\subseteq S}}\nu(P)\langle f\rangle_P=
\int_0^\infty
\sum_{\substack{P\in\Ss,\ P\subseteq S\\\langle f\rangle_P>\lambda}}
\nu(P)\,d\lambda.
\]
For $\lambda>0$, set
\[E_\lambda=\{x\in S:\mathcal M_{\Ss}(f1_S)(x)>\lambda\}.\]
Whenever $P\subseteq S$ and $\langle f\rangle_P>\lambda$, we have $P\subseteq E_\lambda$. Since $\Ss$ is $\Lambda$-Carleson,
\[\sum_{\substack{P\in\Ss,\ P\subseteq S\\\langle f\rangle_P>\lambda}}
\nu(P)
\leq
\sum_{\substack{P\in\Ss\\P\subseteq E_\lambda}}\nu(P)
\leq
\Lambda\nu(E_\lambda).\]
Therefore,
\[
\sum_{\substack{P\in\Ss\\P\subseteq S}}\int_P f(x)d\nu(x)
\leq
\Lambda\int_0^\infty\nu(E_\lambda)\,d\lambda=
\Lambda\int_S \mathcal M_{\Ss}(f1_S)(x)d\nu(x).
\]
\end{proof}

\begin{prop}\label{prop: infCarl}
    Let $\nu\in\mathcal C_\Ss$. We have that $\omega\in A_{\infty}^{\Ss}$ if and only if $\omega d\nu\in\mathcal C_\Ss$.
\end{prop}
\begin{proof}
    Assume that $\omega\in A_\infty^\Ss$. Then applying Lemma \ref{lem:CET} with $f=\omega$, one gets
\begin{equation}\label{eq: inf}
    \sum_{\substack{P\in\Ss\\P\subseteq S}}\omega(P)\leq \Lambda\int_S\mathcal M_\Ss(\omega 1_S)(x)d\nu(x)\leq \Lambda[\omega]_{\infty,\Ss}\omega(S),
\end{equation}
so that $\omega\nu\in\mathcal C_\Ss$. On the other hand, assume $\omega\nu\in\mathcal C_\Ss$ with constant $\Lambda'$. For $S\in\Ss$ and $x\in\X$ we have that 
\[\int_S \mathcal M_\Ss(\omega 1_S)(x)d\nu(x)\leq \sum_{\substack{P\in\Ss\\P\subseteq S}}\omega(P)\leq\Lambda'\omega(S),\]
which proves that $\omega\in A_\infty^\Ss$.
\end{proof}

We finish the one-weight discussion with a weak-type estimate. Its proof closely follows those of Theorems 1.1 and 1.2 in \cite{CNSW2025}. We have chosen to give full details for the case $p=1$ to underline the differences that arise when the measure space may contain point masses. 

\begin{prop}\label{prop:weightsparsop}
We have the following estimates:
    \begin{align*}
        \|\mathcal A_{\Ss}\|_{L^1(\omega)\to L^{1,\infty}(\omega)}
        &\lesssim [\omega]_{1,\Ss}\Bigl(\log[\omega]_{\infty,\Ss}+1\Bigr),&&\omega\in A_1^{\Ss};\\
        \|\mathcal A_{\Ss}\|_{L^p(\omega)\to L^{p,\infty}(\omega)}
        &\lesssim [\omega]_{p,\Ss}^{\frac 1p}[\omega]_{\infty,\Ss}^{\frac 1{p'}},&&\omega\in A_p^{\Ss},\qquad 1<p<\infty.
    \end{align*}
    Furthermore, for every $1\leq p<\infty$, if $\mathcal A_\Ss$ is bounded from $L^p(\omega)$ to $L^{p,\infty}(\omega)$, then $\omega\in A_p^\Ss$.
\end{prop}

\begin{proof}
    Take $f\in L^1(\omega)$ normalized. By the following classical characterization of the $L^{1,\infty}$ norm, it is sufficient to prove that 
    \[\|\mathcal A_{\Ss}f\|_{L^{1,\infty}(\omega)}\approx \sup_{\substack{E\subseteq \G\\0<\omega(E)<\infty}}\inf_{\substack{E'\subseteq E\\ \omega(E')\geq \frac 12 \omega(E)}}\langle \mathcal A_{\Ss}f,\omega 1_{E'}\rangle \lesssim [\omega]_{1,\Ss}(\log[\omega]_{\infty,\Ss}+1).\]
    Let $E\subseteq \G$ be such that $0<\omega(E)<\infty$ and set $\gamma=2[\omega]_{1,\Ss}\omega(E)^{-1}$. We set $E'=\{x\in E\colon \mathcal M_{\Ss}f(x)\leq \gamma\}$. By the fact that $\|\mathcal M_{\Ss}f\|_{L^{1,\infty}(\omega)}\leq[\omega]_{1,\Ss}\|f\|_{L^1(\omega)}$, we have that $\omega(E')\geq \frac 12\omega(E)$. Clearly, it is sufficient to consider the averages of $f$ on sets of the family $\Ss_+=\{S\in\Ss\colon \avg fS\leq \gamma\}$, otherwise they do not intersect $E'$. For each $\lambda\in (0,1)$, let $\Ss_\lambda=\{S\in \Ss_+\colon \omega(E'\cap S)>\lambda\omega(S)\}$ and $\Ss_\lambda^*$ be the subfamily of $\Ss_\lambda$ made with the maximal sets with respect to inclusion. We have 
    \[\sum_{S\in\Ss_+}\avg fS\avg{\omega 1_{E'}}{S}\nu(S)=\sum_{S\in\Ss_+}\avg fS\omega(S)\int_0^{\frac{\omega(E'\cap S)}{\omega(S)}}d\lambda=\int_0^1\sum_{S_0\in\Ss_\lambda^*}\sum_{\substack{S\in\Ss_\lambda\\S\subseteq S_0}}\avg fS \omega(S)d\lambda.\]
    Now fix $\lambda\in (0,1)$ and $S_0\in\Ss_\lambda^*$, by Hölder inequality for a certain $r>1$ and~\eqref{eq: inf}, we have
    \[ \sum_{\substack{S\in\Ss_\lambda\\S\subseteq S_0}} \avg fS \omega(S)
    \leq
    \Bigg( \sum_{\substack{S\in\Ss_\lambda\\S\subseteq S_0}} \avg fS^r \omega(S)\Bigg)^{\frac 1r} \Bigg(\sum_{\substack{S\in\Ss_\lambda\\S\subseteq S_0}} \omega(S)\Bigg)^{\frac 1{r'}}
    \lesssim [\omega]_{\infty,\Ss}^{\frac 1{r'}}\omega(S_0)^{\frac 1{r'}}\Bigg(\sum_{\substack{S\in\Ss_\lambda\\S\subseteq S_0}} \avg fS^r \omega(S)\Bigg)^{\frac 1r} .\]
    Let $C\in (1,2)$ we decompose 
    \[ \Ss_\lambda=\bigsqcup_{k=0}^{\infty}\mathcal E_k,\qquad \mathcal E_k:=\{S\in\Ss_\lambda\colon C^{-(k+1)}\gamma<\avg fS\leq C^{-k}\gamma\}. \]
    Let $N\in\mathbb N$ be as in Lemma~\ref{lem: equiv Carl} for $\eta=1/2$. For each $S\in\mathcal E_k$, we consider 
    \[F_N(S)=\{x\in S\colon \#\Ss_{k,S}(x)>N\},\qquad\text{where}\qquad \Ss_{k,S}(x)=\{P\in\mathcal E_k\colon x\in P\subseteq S \}.\]
    Clearly, we have that $F_N(S)$ is a union of sets of $\mathcal E_k$. 
    We set
    \[ E(S)=S\setminus\bigcup_{P\in\mathcal F_N(S)}P,\qquad\text{where}\qquad\mathcal F_N(S)=\{P\in\mathcal E_k\colon P\subseteq F_N(S)\text{ maximal}\}.\]
    By using the fact that $\nu(F_N(S))\leq \frac 12\nu(S)$ (similarly to proof of Lemma~\ref{lem: equiv Carl}), we have
    \[\sum_{P\in\mathcal F_N(S)}\int_Pf(x)d\nu(x)\leq C^{-k}\gamma \sum_{P\in\mathcal F_N(S)}\nu(P)\leq \frac 12C^{-k}\gamma \nu(S)\leq  \frac 12C \int_S f(x)d\nu(x),
    \]
    so that, by $C<2$
    \[\int_Sf(x)d\nu(x)\lesssim \int_{E(S)}f(x)d\nu(x). \]
    Furthermore, we have that the family $\{E(S)\}_{S\in\mathcal E_k}$ is pairwise almost disjoint. Indeed, for each $x\in E$, if $\#\{S\in\mathcal E_k\colon x\in E(S)\}>N$, then we have a contradiction of the definition of $\mathcal{F}_N(S)$. We now use all these elements to show that
    \begin{align*}
        \sum_{\substack{S\in\Ss_\lambda\\S\subseteq S_0}}\avg fS\omega(S)\lesssim &[\omega]_{\infty,\Ss}^{\frac 1{r'}}\omega(S_0)^{\frac 1{r'}}\Bigg(\sum_{k=0}^{\infty}\sum_{\substack{S\in\mathcal E_k\\S\subseteq S_0}} \avg fS^r \omega(S)\Bigg)^{\frac 1r}\\
        \leq &\gamma^{\frac 1{r'}}[\omega]_{\infty,\Ss}^{\frac 1{r'}}\omega(S_0)^{\frac 1{r'}}\Bigg(\sum_{k=0}^{\infty}C^{-k(r-1)} \sum_{\substack{S\in\mathcal E_k\\S\subseteq S_0}} \avg fS \omega(S)\Bigg)^{\frac 1r}\\
        \lesssim&\gamma^{\frac 1{r'}}[\omega]_{\infty,\Ss}^{\frac 1{r'}}\omega(S_0)^{\frac 1{r'}}\Bigg(\sum_{k=0}^{\infty}C^{-k(r-1)} \sum_{\substack{S\in\mathcal E_k\\S\subseteq S_0}} \int_{E(S)}f(x)\mathcal M_{\Ss}(\omega 1_{S_0})(x)d\nu(x)\Bigg)^{\frac 1r}\\
        \lesssim&r'\gamma ^{\frac 1{r'}}[\omega]_{\infty,\Ss}^{\frac 1{r'}}\omega(S_0)^{\frac 1{r'}}\Bigg(\int_{S_0}f(x)\mathcal M_{\Ss}(\omega 1_{S_0})(x)d\nu(x)\Bigg)^{\frac 1r},
    \end{align*}
    where we used that $(1-C^{1-r})^{-\frac 1r}\approx r'$ and that the $\{E(S)\}_{S\in\mathcal E_k}$ are pairwise almost disjoint. In particular, we stress that the constant arising in the control of the sum of $1_{E(S)}$, $S\in\mathcal E_k$, does not depend on the measure $\omega$, but on the geometry of $\Ss$. Now we conclude by passing to the sum over all $S_0\in \Ss_\lambda^*$, and by Hölder, we control
     \begin{align*}
         \langle \mathcal A_{\Ss}f,\omega 1_{E'}\rangle\lesssim&r'\gamma ^{\frac 1{r'}}[\omega]_{\infty,\Ss}^{\frac 1{r'}}\int_0^1
        \Bigg( \sum_{S_0\in\Ss_\lambda^*}\omega(S_0)\Bigg)^{\frac 1{r'}}\Bigg(\sum_{S_0\in\Ss_\lambda^*}\int_{S_0}f(x)\mathcal M_{\Ss}(\omega 1_{S_0})(x)d\nu(x)\Bigg)^{\frac 1r}d\lambda\\
        \leq&r'\gamma ^{\frac 1{r'}}[\omega]_{\infty,\Ss}^{\frac 1{r'}}
        \|f\|^{\frac 1r}_{L^1(\mathcal M_{\Ss}\omega)} \int_0^1 \lambda^{-\frac 1{r'}} \Bigg(\sum_{S_0\in\Ss_\lambda^*}\omega(E'\cap S_0)\Bigg)^{\frac 1{r'}}d\lambda\\
        \leq&rr'\gamma ^{\frac 1{r'}}[\omega]_{\infty,\Ss}^{\frac 1{r'}}\omega(E)^\frac{1}{r'}\|f\|^{\frac 1r}_{L^1(\mathcal M_{\Ss}\omega)}
        \leq 2^{\frac 1{r'}} rr'[\omega]_{\infty,\Ss}^{\frac 1{r'}}[\omega]_{1,\Ss}.
    \end{align*}
    We conclude by taking $r'=\log[\omega]_{\infty,\Ss}+2$.

    For the converse implication, assume first that $1<p<\infty$ and fix $S\in\Ss$. By testing the weak-type estimate with $f=\omega^{-\frac1{p-1}}1_S$, we obtain
\[
\left\langle\omega^{-\frac1{p-1}}\right\rangle_S^p
\omega(S)
\leq
\|\mathcal A_\Ss f\|^p_{L^{p,\infty}(\omega)}
\lesssim
\|f\|_{L^p(\omega)}^p
=
\left\langle\omega^{-\frac1{p-1}}\right\rangle_S \nu(S).
\]
Hence $\langle\omega\rangle_S
\left\langle\omega^{-\frac1{p-1}}\right\rangle_S^{p-1}
\lesssim1$, uniformly in $S\in\Ss$, and therefore $\omega\in A_p^\Ss$.
For $p=1$, fix $S\in\Ss$ and $x\in S$, and test the weak-type estimate with $f=\delta_x\nu(x)^{-1}$. Since $\mathcal A_\Ss f\geq\nu(S)^{-1}1_S$ and
$\|f\|_{L^1(\omega)}=\omega(x)$, we obtain
$\avg\omega S\lesssim\omega(x)$. Taking the supremum of $\omega(x)^{-1}$ over $x\in S$ gives
$\avg\omega S\|\omega^{-1}1_S\|_\infty\lesssim 1$,
and hence $\omega\in A_1^\Ss$.
\end{proof}

We next turn to two-weight inequalities. We are interested in sufficient conditions on a pair of weights $\sigma,\omega$ that guarantee that $\mathcal A_{\Ss}(\cdot\sigma)\colon L^p(\sigma)\to L^p(\omega)$ is bounded. Our starting point is a classical result in a version that is more convenient for us.
\begin{definition}
    Given $1<p<\infty$, we define the mixed class of weights
\[A_p^{\Ss,2}(\nu)=\Big\{(\omega,\sigma)\in A_\infty^\Ss(\nu)\times A_\infty^\Ss(\nu)\colon [\omega,\sigma]_{p,\Ss}:=\sup_{Q\in \Ss}\avg{\omega}{Q}\avg{\sigma}{Q}^{p-1}<+\infty\Big\}.\]
\end{definition}
\begin{prop}[\cite{Hytonen2012,LSUT2009,FH2018}]\label{prop: 2pesi}
Let $\nu\in\mathcal C_{\Ss}$. Let $1<p<\infty$ and $(\omega,\sigma)\in A_p^{\Ss,2}(\nu)$. We have that
\begin{align*}
    \| \mathcal{A}_{\Ss}(\cdot\,\sigma)\| _{L^{p}(\sigma)\to L^{p}(\omega)}
    \lesssim&[\omega,\sigma]_{p,\Ss}^{\frac 1p}\big([\sigma]_{\infty,\Ss}^{\frac 1p}+[\omega]_{\infty,\Ss}^{\frac 1{p'}}\big).
\end{align*}
    Furthermore, if $\omega,\sigma\in A_\infty^\Ss(\nu)$ are such that $\mathcal A_\Ss(\cdot\,\sigma)$ is bounded from $L^p(\sigma)$ to $L^p(\omega)$, then $(\omega,\sigma)\in A^{\Ss,2}_{p}(\nu)$.
\end{prop}

In particular, it follows that, if $\omega\in A_p^\Ss(\nu)$ then $\As$ is bounded on $L^p(\omega)$, with
    \begin{equation*}
        \|\As\|_{L^p(\omega)\to L^p(\omega)}
        \lesssim [\omega]_{p,\Ss}^{\frac 1p}\big([\omega^{1-p'}]_{\infty,\Ss}^{\frac 1p}+[\omega]_{\infty,\Ss}^{\frac 1{p'}}\big)\leq[\omega]_{p,\Ss}^{\max\{1,\frac{1}{p-1}\}},
    \end{equation*}
by setting $\sigma=\omega^{-\frac 1{p-1}}$ and by the fact that $[\omega^{-\frac 1{p-1}}]_{\infty,\Ss}\lesssim [\omega^{-\frac 1{p-1}}]_{p',\Ss}=[\omega]_{p,\Ss}^{\frac 1{p-1}}$, as well as $[\omega]_{\infty,\Ss}\lesssim[\omega]_{p,\Ss}$.

We end this section with the extremely simple sparse domination principle that will be used in the sequel. Assume $T$ is an integral operator with kernel $K_T$, so that it admits the representation
$$
Tf(x) = \int K_T(x,y) f(y) \; d\nu(y),
$$
in a suitable sense. Such suitable sense can be a delicate issue if $\G$ has a complicated structure, but in our discrete applications this will not be a problem and the interpretation of the equality above will be clear. Denote by $K_{\As}$ the kernel of the sparse operator $\As$, that is,
$$
K_{\As}(x,y)= \sum_{x,y \in S \in \Ss} \frac{1}{\nu(S)}.
$$
Then it is obvious that
\begin{equation}\label{eq:condKernel}
|K_T(x,y)| \lesssim K_{\As}(x,y), \quad K_T(x,y) \sim K_{\As}(x,y),
\end{equation}
(with $K_T$ positive in the second comparison above) respectively imply
\begin{equation} \label{eq:conclusionKernel}
|Tf(x)| \lesssim \As f(x), \quad Tf(x) \sim \As f(x).   
\end{equation}

    
\section{Integral operators on trees} \label{sec:sec2}
In this section we provide pointwise sparse domination results from above and below for a class of integral operators on a tree $\X$ whose kernels satisfy certain size conditions. We subsequently apply these results to particular settings in which such conditions arise naturally. A tree $\X$ is a connected discrete manifold with no cycles. Formally, there is a family $\mathfrak E\subseteq \{\{x,y\}\in\X\times\X\colon x\neq y\}$ of couples, called edges, and we say that $x,y\in \X$ are neighbors if $\{x,y\}\in \mathfrak E$ and we write $x\sim y$.  Considering two points, here called vertices, there exists a unique geodesic -- a finite path of connected vertices with no repetitions -- joining them. The tree is endowed with the metric $d$ that counts the edges in the geodesic joining the two vertices. We denote by $B(x,N)=\{y\in \X\colon d(x,y)\leq N\}$ and $S(x,N)=\{y\in \X\colon d(x,y)= N\}$. 
The branching function of $\X$ is the function $q\colon \X\to \mathbb N$ defined by $q(x)=\#\{y\in \X\colon x\sim y\}-1$. We shall always assume  $\X$ to be locally finite -- i.e. $q(x)<+\infty$ for every $x\in \X$ -- and, to avoid trivial cases, that $q\geq 2$.

We say that two geodesic rays -- i.e. infinite geodesic paths -- $\{x_k\}_{k\in \N},\{y_j\}_{j\in \N}\subseteq \X$ are equivalent if they eventually coincide, up to a shift of their parametrizations. In other words, we introduce the following equivalence relation: $\{x_k\}_{k\in \N}\sim\{y_j\}_{j\in \N}$ if and only if there exists $N\in\mathbb N$ and $m\in \mathbb Z$ such that $x_k=y_{k+m}$ for $k\geq N$. The boundary of $\X$ is the quotient set $\partial \X=G(\X)/\sim$ of the set of  geodesic rays $G(\X)$ with the previous relation. Fix a point at infinity $\xi\in \partial \X$. For $x\in \X$, we denote by $[x,\xi)$ the element in $G(\X)$ that starts at $x$ and is equivalent to $\xi$. Every vertex $x\in \X$ has a unique neighbor $x^{(1)}\in\X$, called predecessor, such that $[x^{(1)},\xi)\subseteq [x,\xi)$, all the other neighbors are called successors of $x$: $s(x)=\{y\in \X\colon y\sim x,[x,\xi)\subseteq [y,\xi)\}$. The notion of predecessor can be iterated by defining $x^{(k+1)}=(x^{(k)})^{(1)}$, $k\geq 1$. In this way, by putting $x^{(0)}=x$, we can say that $[x,\xi)=\{x^{(k)}\}_{k\in\N}$. Finally, given two vertices $x,y\in \X$, the confluent of $x$ and $y$ is the unique vertex $x\wedge y$ such that $[x,\xi)\cap [y,\xi)=[x\wedge y,\xi)$. 

Given $x\in \X$, we call sector of $x$ the (infinite) portion of the tree that is reached from $\xi$ passing through $x$, namely
        \[ T_x=\{y\in\X\colon [x,\xi)\subseteq [y,\xi)\}.\]
        We set $\T=\{T_x\colon x\in \X\}$. Notice that, from the definition of $E_Q$ we provide in the proof of Lemma~\ref{lem: equiv Carl}, we have
        $E_{T_x}=T_x\cap B(x,N-1)$.
        This means that a measure lies in $\mathcal C_{\T}$ whenever the measure of a sector is mainly distributed on a triangle of uniform height. We call pruned sector of $x$ the set
        \[U_x=T_x\setminus\{x\},\]
        and we denote the family of all pruned sectors by $\U$. We also use the notation $\Qp=U_x$ when $Q=T_x\in\T$, as well as $x_Q\in \X$ if $Q=T_x\in\T$ and $x_R\in X$ if $R=U_x\in\U$. As before, it is easy to check that $E_{U_x}=U_x\cap B(x,N)$.
The families $\T$ and $\U$ are clearly nested. Hence, we call predecessor of $Q\in\T$ the smaller set $Q^{(1)}\in \T$ that properly contains $Q$. In other words, if $Q=T_x$, then $Q^{(1)}=T_{x^{(1)}}$. Hence, as for vertices, we can iterate the definition, getting $Q^{(k+1)}=(Q^{(k)})^{(1)}$, $k\geq 1$. Similarly, we define the predecessor for the family $\U$.

We set $\mathcal Q=\T\cup\U$. $\mathcal{Q}$ is still nested, by the definitions of $\T$ and $\U$.  Clearly, we have that $\mathcal C_{\mathcal Q}=\mathcal C_\T\cap \mathcal C_\U$, and we claim that
$ \mathcal C_{\T}=\mathcal C_{\U}$. 
Indeed, given $\nu\in\mathcal C_{\T}$ and $ R\in \U$, we have
\[\sum_{\substack{ S\in\U\\  S\subseteq  R}}\nu(S)=\sum_{\substack{ Q\in\T\\ Q\subseteq  R}}\nu(Q)=\sum_{y\in s(x_R)}\sum_{\substack{ Q\in\T\\ Q\subseteq T_y}}\nu(Q)\leq \Lambda \sum_{\substack{P\in\T \\ P^{(1)}=R}}\nu (P)=\Lambda\nu(R). \]
On the other hand, if $\nu \in\mathcal C_{\U}$ and $R\in\T$, then
\[\sum_{\substack{S\in \T\\ S\subseteq R}}\nu(S)=\nu(R)+\sum_{\substack{S\in \T\\ S\subsetneq R}}\nu(S)=\nu(R)+\sum_{\substack{  P\in \U\\   P\subseteq   R}}\nu(  P)\leq \nu(R)+\Lambda\nu(  R)\leq (\Lambda+1)\nu(R). \]
It is then straightforward to conclude that $\mathcal C_{\Q}=\mathcal C_{\T}=\mathcal C_{\U}$.

The family of sectors can be regarded as the discrete counterpart of Carleson boxes in the disk and in the upper half-plane. Both are naturally viewed as tents over subsets of the boundary. In the upper half-plane, the Carleson box $Q_I=I\times(0,|I|]$ is the tent over a boundary interval $I$, and its measure with respect to $y^{\alpha-2}dxdy$, $\alpha>1$, is proportional to $|I|^\alpha$. Similarly, on a $q$-homogeneous tree $X$, if $\mu_\alpha(x)=q^{-\alpha|x|}$, $x\in \X$, and $|\cdot|$ denotes the canonical boundary measure on $\partial \X$, then $\mu_\alpha(T_x)=c_\alpha |\partial T_x|^\alpha$, for some $c_\alpha>0$. Thus, sectors reproduce both the boundary geometry and the natural scaling of Carleson boxes. Likewise, pruned sectors in $\U$ correspond to the lower part of a Carleson box, since in our discrete setting the upper part reduces to the generator of the sector.

In what follows we shall use $\T$ and $\U$ as the natural discrete geometries for two Bergman-type projectors arising in different tree settings. Indeed, in the continuous setting, Carleson boxes arise naturally in the study of the Bergman projector; see, for instance,~\cite{PR2013}. In the same spirit, the classes $A_p^\T$ correspond to the classical Békollé-Bonami classes~\cite{BB1978}, and inspired by this we shall write $B_p^\T=A_p^\T$.  We also set $B_p^\U=A_p^\U$, $B_p^\Q=A_p^\Q$ and $B_p^{\Q,2}=A_p^{\Q,2}$.

\

\

\noindent\textbf{Horocyclic trees and Bergman-type measures.}  
We first work in the extension of \cite{DMMR2025} to nonhomogeneous trees -- similar to the one from \cite{DMMV2023} to \cite{CADMMRV2025} --, that we briefly describe. The horocyclic index of a vertex $x\in\X$, with respect to the boundary point $\xi\in\partial\X$, is defined by
\[
|x|=d(x,o)-2d(o,x\wedge o).
\]
Furthermore, we shall use the notation $|Q|=|x_Q|$ for $Q\in\T$. When $q(x)=q(|x|)$, we say that $\X$ is horocyclic. Given such $\X$ and $\alpha>1$, we consider the (horocyclic) measure given by 
    \begin{equation}\label{eq: def mu}
        \mu_\alpha(o)=1,\qquad\text{and}\qquad \mu_\alpha(x^{(1)})=q(x^{(1)})^\alpha\mu_\alpha(x). 
    \end{equation}
For $x\in\X$ we have
\begin{align}\label{eq: mu set}
    \mu_\alpha(T_x)&=\sum_{m=0}^\infty \mu_\alpha(T_x\cap S(x,m))=\mu_\alpha(x)\sum_{m=0}^\infty \prod_{\ell=|x|}^{m+|x|-1} q(\ell)^{1-\alpha}\nonumber\\
    &\leq \mu_\alpha(x)\sum_{m=0}^\infty 2^{(1-\alpha)m}\lesssim \mu_\alpha(x),
\end{align}
and so by Lemma~\ref{lem: equiv Carl} $\mu_\alpha\in\mathcal C_\T=\mathcal C_\U=\mathcal C_\Q$. A function $f\colon \X\to \C$ is said to be harmonic at $x\in\X$ whenever
\[f(x)=\frac 1{q(x)+1}\sum_{y\sim x}f(y).\]
We say that $f$ is harmonic if it is harmonic at every vertex $x$. For $1\leq p<\infty$, we denote $L^p_\alpha=L^p(\mu_\alpha)$ and $\mathcal B_\alpha^p$ its subspace of harmonic functions. To identify the kernel of $\Proj$, which we denote $K_\alpha$, we do the following: given $Q\in\T$, we consider an orthonormal basis $\{e_Q^j\}_{j=1}^{q(x_Q)-1}$ of the space 
\[W_Q=\Bigl\{\psi\colon s(x_Q)\to \C\colon \sum_{v\in s(x_Q)}\psi(v)=0\Bigr\},\]
endowed with the scalar product with respect to the counting measure on $s(x_Q)$.
For every $1\leq j<q(x_Q)$, we set
\begin{equation*}
    h_Q^j=\sum_{v\in s(x_Q)}e_Q^j(v)\varphi_{|Q|}1_{T_v}
\end{equation*}
where 
\[ \varphi_k(x)=\sum_{j=1}^{|x|-k}\prod_{\ell=k+1}^{|x|-j}\frac 1{q(\ell)}\in [1,2).\]
Adapting easily from \cite{CADMMRV2025,DMMR2025,DMMV2023}, one can see that 
$\{h_Q^j\colon Q\in \T,1\leq j<q(x_Q)\}$
is a complete orthogonal system and that, if we fix $v\in s(x_Q)$,
\[\|h_Q^j\|_{2,\alpha}^2 
=\|\varphi_{|Q|}1_{T_v}\|_{2,\alpha}^2\sum_{u\in s(x_Q)}|e_Q^j(u)|^2
=\|\varphi_{|Q|}1_{T_v}\|_{2,\alpha}^2\approx \mu_\alpha(T_v).\]
For simplicity, since $\|h_Q^j\|_{2,\alpha}^2$ does not depend on $j$, we introduce the notation $E_\alpha(Q)=\|h_Q^j\|_{2,\alpha}^2$. We have $E_\alpha(R^{(1)})\approx \mu_\alpha(R)$. Then,
\begin{align*}
    K_\alpha&=\sum_{Q\in\T}K_{Q,\alpha},\qquad\text{where}\qquad K_{Q,\alpha}(x,y)=\frac{1}{E_\alpha(Q)}\sum_{j=1}^{q(x_Q)-1}h_Q^j(x)h_Q^j(y),\quad Q\in \T.
\end{align*}
Each $W_Q$ is a reproducing kernel Hilbert space, with kernel given by
$k_Q(u,v)=\delta_{uv}-q(x_Q)^{-1}$. Hence for $x,y\in \dot Q$, if we set $v_x,v_y\in s(x_Q)$ to be such that $x\in T_{v_x}$ and $y\in T_{v_y}$, we have
\begin{align}\label{eq: hh}
    \Bigg|\sum_{j=1}^{q(x_Q)-1}h_Q^j(x)h_Q^j(y)\Bigg| &=\Bigg|\sum_{j=1}^{q(x_Q)-1}e_Q^j(v_x)e_Q^j(v_y)\varphi_{|Q|}(x)\varphi_{|Q|}(y)\Bigg|\nonumber\\
    &=|k_Q(v_x,v_y)|\varphi_{|Q|}(x)\varphi_{|Q|}(y)\approx\begin{dcases*}
    1,&\text{if }$x\wedge y\neq x_Q$,\\
    q(x_Q)^{-1},&\text{if }$x\wedge y=x_Q$.
\end{dcases*}
\end{align}

The corresponding positive Bergman operator $\Proj^+$ is the integral operator with kernel
\[
K_\alpha^+(x,y):=\sum_{Q\in\T}|K_{Q,\alpha}(x,y)|
=\sum_{Q\in\T}\frac{1}{E_\alpha(Q)}\Bigg|\sum_{j=1}^{q(x_Q)-1}h_Q^j(x)h_Q^j(y)\Bigg|.
\]
Clearly, for $f\geq 0$, we have $|\Proj f|\leq \Proj^+ f$. 

\begin{rem}
We first recall that, in the continuous setting of the upper half-plane, the positive Bergman projector is obtained by replacing the Bergman kernel with its absolute value. Although the kernel exhibits oscillatory behaviour, it is complex-valued, and its modulus satisfies precise lower estimates. The situation is substantially different in our setting. Indeed, $K_\alpha$ is real-valued and oscillatory, and therefore its changes of sign force it towards zero. Since the underlying space is discrete, the kernel need not actually vanish, but its absolute value may nevertheless become arbitrarily small. As a consequence, $|K_\alpha|$ does not satisfy meaningful lower estimates of the desired type, and one is naturally led to replace it with a larger positive kernel.
Furthermore, the cancellation properties of $K_\alpha$ are encoded in each kernel $K_{Q,\alpha}$, $Q\in\T$, considered as a whole, rather than in the individual components $\langle f,h_Q^j\rangle h_Q^j$, whose number within a given cube $Q$ may be arbitrarily large. Taking the absolute value after summing over $j$, rather than term by term, therefore avoids introducing a dependence on the number of elements associated with $Q$. This is precisely what allows $\Proj^+$ to enjoy the correct size and to recover the upper estimates satisfied by $\Proj$. The distinction becomes essential when the branching function $q$ is unbounded.
\end{rem}

\begin{theorem}
    The following estimates for $\Proj^+$, and hence for $\Proj$, hold true:
    \begin{align*}
        \| \Proj^+(\cdot\sigma)\| _{L^{p}_\alpha(\sigma)\to L^{p}_\alpha(\omega)}
&\lesssim[\omega,\sigma]_{p,\Q}^{\frac 1p}\big([\sigma]_{\infty,\Q}^{\frac 1p}+[\omega]_{\infty,\Q}^{\frac 1{p'}}\big)&&\text{if }(\omega,\sigma)\in B^{\Q,2}_{p},\\
    \|\Proj^+\|_{L^1_\alpha(\omega)\to L^{1,\infty}_\alpha(\omega)}
        &\lesssim [\omega]_{1,\Q}\Bigl(\log[\omega]_{\infty,\Q}+1\Bigr)&&\text{if }\omega\in B_1^\Q.
    \end{align*}
    Furthermore, if $\Proj^+(\,\cdot\,\sigma)\colon L_\alpha^{p}(\sigma)\to L_\alpha^{p}(\omega)$ is bounded with $\omega,\sigma\in B_\infty^\Q$, then $(\omega,\sigma)\in  B^{\Q,2}_{p}$, and if $\Proj^+\colon L_\alpha^{1}(\omega)\to L_\alpha^{1,\infty}(\omega)$ is bounded, then $\omega\in B_1^\Q$.
    \end{theorem}
    \begin{proof}
        It is sufficient to show that the kernel $K_\alpha^+$ satisfies \eqref{eq:condKernel} with respect to $K_{\As}$ for an appropriately chosen family $\Ss$. If we take $\Ss=\Q$, this means that we have to show
        \begin{equation}\label{eq: cond K+}
        K_\alpha^+(x,y)\approx \begin{dcases*}
            \mu_\alpha(T_{x\wedge y})^{-1},&\text{if }$x\wedge y\in\{x,y\}$\\
            \mu_\alpha(U_{x\wedge y})^{-1},&\text{if }$x\wedge y\not\in\{x,y\}$,
        \end{dcases*}
        \end{equation}
        by Proposition~\ref{prop: somma sect}. The estimates for $\Proj$ will then follow from those of $\Proj^+$. Let $z\in\X$. Using that $E_\alpha(T_z^{(1)})\approx \mu_\alpha(T_z)$ and~\eqref{eq: hh}, for $x,y\in U_z^{(1)}$ we have 
        \begin{align}\label{eq: K_Q stime}
            |K_{T_z^{(1)},\alpha}(x,y)|\approx \begin{dcases*}
                \mu_\alpha(T_z)^{-1},&if $x\wedge y\neq z^{(1)}$,\\
                \mu_\alpha(U_z^{(1)})^{-1},&if $x\wedge y=z^{(1)}$.
            \end{dcases*}
        \end{align}
    Support considerations about $K_{Q,\alpha}$ yield
    \[K_\alpha^+(x,y)=|K_{T_{x\wedge y},\alpha}(x,y)|+\sum_{k=1}^{\infty}|K_{T_{x\wedge y}^{(k)},\alpha}(x,y)|,\]
    where the first term vanishes when $x\wedge y\in\{x,y\}$. Hence, by using~\eqref{eq: K_Q stime} we derive
    \[K_\alpha^+(x,y)\approx\begin{dcases*}
    \sum_{k=1}^{\infty}\mu_\alpha(T_{x\wedge y}^{(k-1)})^{-1},&\text{if }$x\wedge y\in\{x,y\}$,\\
        \mu_\alpha(U_{x\wedge y})^{-1}+\sum_{k=1}^{\infty}\mu_\alpha(T_{x\wedge y}^{(k-1)})^{-1},&\text{if }$x\wedge y\not\in\{x,y\}$,
    \end{dcases*}\]
    which by Proposition \ref{prop: somma sect} is what we wanted.
    \end{proof}

When $q$ is bounded, one has $|\Proj f| \lesssim \mathcal{A}_{\T}|f|$, see \cite{CADMMRV2025,DMMR2025,DMMV2023}. This choice is geometrically natural, since in this regime the Gromov metric, whose balls are precisely the sectors, turns the underlying metric measure space into a doubling one. The situation changes drastically when $q$ is unbounded: the doubling property is lost, and the example below shows that domination by sectors alone is no longer sufficient. More precisely, an additional term is genuinely needed, and it is exactly the one arising from the domination associated with the family $\U$. In this sense, the example shows that the two separate conditions appearing in~\eqref{eq: cond K+} are not merely technical artifacts, but are both necessary in the nondoubling setting.
Indeed, let $v\in\X$, consider the function $f_v=\delta_v\mu_\alpha(v)^{-1}\in L^1_\alpha$ and note that
\[\mathcal A_{\T}f_v(x)=\sum_{k=0}^{\infty}\avg{f_v}{T_{x\wedge v}^{(k)}}=\sum_{k=0}^{\infty}\frac 1{\mu_\alpha(T_{x\wedge v}^{(k)})}\lesssim \frac 1{\mu_\alpha(T_{x\wedge v})}, \]
by Proposition~\ref{prop: somma sect}.
Hence, for $u\in\X$, $u\neq v$ and $u^{(1)}=v^{(1)}$, by the fact that $K_\alpha^+$ satisfies~\eqref{eq: cond K+}, we have
\[\frac{\Proj^+f_v(u)}{\mathcal A_\T f_v(u)}\gtrsim K_\alpha^+(u,v)\mu_\alpha(T_v^{(1)})\gtrsim \frac{\mu_\alpha(T_v^{(1)})}{\mu_\alpha(U_v^{(1)})}\gtrsim q(v^{(1)})^{\alpha-1}.\]
Whenever $\|q\|_\infty=\infty$, we hence have that 
\[ \sup_{v\in\X}\left\|\frac{\Proj^+f_v}{\mathcal A_\T f_v}\right\|_{\alpha,\infty}=+\infty.\]
In an analogous way, it can be shown that $\mathcal A_\U$ would not be sufficient on its own.

\begin{rem}
    Let $\varphi\in L_\alpha^{\infty}$ and consider the associated Toeplitz operator
$T_{\varphi}f=\Proj(f\varphi)$. Notice that, if we take a weight $\omega$ and consider $\sigma=\omega^{-\frac{1}{p-1}}\varphi^{p'}$,
then
\begin{align*}
    \| T_{\varphi}f\|_{L_{\alpha}^{p}(\omega)}&= \|\Proj((f\varphi\sigma^{-1})\sigma)\|_{L_{\alpha}^{p}(\omega)}
    \\
    &\leq
    \| \Proj(\cdot\sigma)\|_{L_{\alpha}^{p}(\sigma)\to L_{\alpha}^{p}(\omega)}\|f\varphi\sigma^{-1}\| _{L_{\alpha}^{p}(\sigma)}
    =\|\Proj(\cdot\sigma)\|_{L_{\alpha}^{p}(\sigma)\to L_{\alpha}^{p}(\omega)}\| f\| _{L_{\alpha}^{p}(\omega)}.
\end{align*}
This means that
\begin{align*}
    \|T_\varphi\|_{L^p_\alpha(\omega)\to L^p_\alpha(\omega)}
    \leq&\| \Proj(\cdot\sigma)\|_{L_{\alpha}^{p}(\sigma)\to L_{\alpha}^{p}(\omega)}\\
    \lesssim & [\omega,(\omega^{-\frac 1p}\varphi)^{p'}]_{p,\Q}^{\frac 1p}\bigg(\big[(\omega^{-\frac 1p}\varphi)^{p'}\big]_{\infty,\Q}^{\frac 1p}+\big[\omega\big]_{\infty,\Q}^{\frac 1{p'}}\bigg).
\end{align*}
In the specific case in which $\varphi=\mu_{b}$ for some $b\geq0$,
notice that $T_{\mu_{b}}\colon L^{p}_\alpha\to L^{p}_\alpha$ if
\begin{align*}
[1,\mu_{bp'}]_{p,\Q}&=\sup_{Q\in\Q}\Bigg(\frac{\mu_{\alpha+bp'}(Q)}{\mu_{\alpha}(Q)}\Bigg)^{p-1}<+\infty,
\end{align*}
which is finite if and only if $\alpha+bp'>1$, i.e. $bp>(1-\alpha)(p-1)$.
This result extends the known result for Toeplitz operators on the disk (see~\cite{Zhu2007}) to the nondoubling discrete setting.
\end{rem}

\noindent\textbf{Flow measures on trees and forward Laplacians.} A second application of our sparse domination principle is obtained by considering the operators introduced in~\cite{OS2026}. Let $\X$ be any locally finite tree, and let $m$ be a flow measure on $\X$. This means that
$m(x)=m(s(x))$ for every $x\in \X$. We say that a function is $m$-harmonic if
\[ f(x)=\sum_{y\in s(x)}\frac{m(y)}{m(x)}f(y) \mbox{ holds for all } x\in\X.\]
We stress that the harmonicity here only takes into account the downward direction of the flow, because the value of $f$ in $x^{(1)}$ is not involved in the average. The corresponding Laplacian is therefore a differential operator of order $1$, unlike the usual one. Let $\sigma\in L^1(\mathbb Z)$ be nonnegative. We define the measure 
\[m\sigma(x):=m(x)\sigma(|x|),\qquad x\in\X. \]
It is easy to see that, for every $x\in\X$, we have 
    \begin{equation}
        \label{eq: sect ms} m\sigma(T_x)=m(x)S_{|x|},\qquad m\sigma(U_x)=m(x)S_{|x|+1}\qquad \text{where}\qquad S_N:=\sum_{n=N}^{\infty}\sigma(n).
    \end{equation}
We consider the $m$-Bergman space $\mathcal B_m^2=\{f\in L^2(m\sigma)\colon f\text{ is $m$-harmonic}\}$ and the $m$-Bergman orthogonal projector $\Proj_m\colon L^2(m\sigma)\to \mathcal B_m^2$. The authors of~\cite{OS2026} prove that $\sigma\in L^1(\mathbb Z)$ is necessary to have a nonempty subspace of harmonic functions within $L^2(m\sigma)$, and that $\Proj_m$ extends to a bounded operator on $L^p(m\sigma)$ for every $1<p<\infty$ and is of weak type (1,1).  For a certain class of measures $m\sigma$ we can say more.

\begin{rem}
    We have $m\sigma\in \mathcal C_\T=\mathcal C_\U=\mathcal C_\Q$ if and only if there exists $N\in\mathbb N$ and $0<c<1$ such that
    \begin{equation}\label{eq: S simili}
        S_{M+N}\leq c S_M,\qquad M\in\mathbb Z.
    \end{equation}
    Indeed, by the proof of Lemma~\ref{lem: equiv Carl}, saying that $m\sigma\in\mathcal C_\T$ is equivalent to having that for some $N\in\N$, 
    \begin{equation}\label{eq: ms Carl}
        m\sigma(T_x)\lesssim m\sigma(T_x\cap B(x,N-1)),\qquad x\in \X.
    \end{equation} 
    Observe that for $x\in\X$
    \begin{equation*}
        m\sigma(T_x\cap B(x,N-1))=\sum_{k=0}^{N-1}\sigma(|x|+k)\sum_{\substack{y\in T_x\\|y|=|x|+k}}m(y)=m(x)(S_{|x|}-S_{|x|+N}),
    \end{equation*} 
    which, together with~\eqref{eq: S simili}, gives~\eqref{eq: ms Carl}.
\end{rem}

In what follows, we assume $m\sigma\in\mathcal C_{\T}$. Otherwise, $L^p$-boundedness of $\mathcal A_{\Q}$ is no longer guaranteed and our method does not provide new weighted inequalities. It remains an interesting open question whether such weighted inequalities can be obtained through a different procedure. In~\cite{OS2026}, it is shown that
        \[\mathcal{P}_mf(x)=\sum_{y\in\X} K_{\mathcal{P}_m}(x,y)f(y)m\sigma(y),\qquad K_{\mathcal{P}_m}(x,y)=\sum_{n=0}^{\infty}\frac{\Psi(x^{(n+1)},x^{(n)},y)}{S_{|x|-n}},\]
    where 
    \begin{equation}\label{eq: def Psi}
        \Psi(v,x,y)=\sum_{u\in s(v)}\frac 1{m(u)}1_{T_u^2}(x,y)-\frac 1{m(v)}1_{U_v^2}(x,y).
    \end{equation}
    In analogy with the horocyclic case, we introduce the operator 
        \[\mathcal{P}^+_mf(x)=\sum_{y\in\X} K_{\mathcal{P}_m}^+(x,y)f(y)m\sigma(y),\qquad K_{\mathcal{P}_m}^+m(x,y)=\sum_{n=1}^{\infty}\frac{|\Psi(x^{(n)},x^{(n-1)},y)|}{S_{|x|-n+1}}.\]

    We use our argument, together with the explicit formulation of the projector to prove the following result.
    \begin{theorem}
       Assume $m\sigma\in\mathcal C_{\Q}$. The following estimates for $\mathcal{P}_m^+$, and hence for $\mathcal{P}_m$, hold true:
    \begin{align*}
        \| \mathcal{P}_m^+(\cdot\tau)\| _{L_{m\sigma}^{p}(\tau)\to L_{m\sigma}^{p}(\omega)}
&\lesssim[\omega,\tau]_{p,\Q}^{\frac 1p}\big([\tau]_{\infty,\Q}^{\frac 1p}+[\omega]_{\infty,\Q}^{\frac 1{p'}}\big)&&\text{if }(\omega,\tau)\in B^{\Q,2}_{p}(m\sigma),\\
    \|\mathcal{P}_m^+\|_{L_{m\sigma}^1(\omega)\to L_{m\sigma}^{1,\infty}(\omega)}
         &\lesssim [\omega]_{1,\Q}\Bigl(\log[\omega]_{\infty,\Q}+1\Bigr)&&\text{if }\omega\in B_1^\Q(m\sigma).
        \end{align*}
        Furthermore, assume
        \begin{equation}\label{eq: ipo Km+}
            \sup_{x\in\X}\frac{m(x)}{m(x^{(1)})}<1.
        \end{equation}
        If $\mathcal{P}_m^+(\,\cdot\,\tau)\colon L_{m\sigma}^{p}(\tau)\to L_{m\sigma}^{p}(\omega)$ is bounded with $\omega,\tau\in B_\infty^\Q(m\sigma)$, then $(\omega,\tau)\in B^{\Q,2}_{p}(m\sigma)$, and if $\mathcal{P}_m^+\colon L_{m\sigma}^{1}(\omega)\to L_{m\sigma}^{1,\infty}(\omega)$ is bounded, then $\omega\in B_1^\Q(m\sigma)$.
    \end{theorem}
    \begin{proof}
        We prove that $K_{\mathcal{P}_m}^+$ satisfies~\eqref{eq: cond K+}. Take $x,y\in\X$ and observe that $\Psi(x^{(n)},x^{(n-1)},y)=0$ whenever $1\leq n< |x|-|x\wedge y|$, by the definition of $\Psi$. Hence
        \[K_{\mathcal{P}_m}^+(x,y)=\frac{|\Psi(x\wedge y,x^{|x|-|x\wedge y|-1)},y)|}{S_{|x\wedge y|+1}}+\sum_{n=|x|-|x\wedge y|+1}^{\infty}\frac{|\Psi(x^{(n)},x^{(n-1)},y)|}{S_{|x|-n+1}}.\]
        Note that the first term vanishes whenever $x\wedge y\in\{x,y\}$, and is otherwise controlled from above and below by $m\sigma(U_{x\wedge y})$, by~\eqref{eq: def Psi} and~\eqref{eq: sect ms}. We wrote it with a small abuse of notation in the case $x=x\wedge y$. Each of the terms of the series can be controlled by the inverse of the measure of the respective sector, always by~\eqref{eq: sect ms}, yielding
        \begin{align*}
        K_{\mathcal{P}_m}^+(x,y)& \lesssim \frac{1}{m\sigma(U_{x\wedge y})}1_{x\wedge y \in \{x,y\}}(x,y)+\sum_{k=0}^{\infty}\frac{1}{m\sigma(T_{x\wedge y}^{(k)})} \\
        & \approx \frac{1}{m\sigma(U_{x\wedge y})}1_{x\wedge y \in \{x,y\}}(x,y)+\frac{1}{m\sigma(T_{x\wedge y})},
        \end{align*}
        by Proposition~\ref{prop: somma sect}. If the measure $m$ satisfies~\eqref{eq: ipo Km+}, then we can also get the estimate from below for $K_{\mathcal{P}_m}^+$, by the fact that, for $n>|x\wedge y|-|x|$,
        \[\frac{|\Psi(x^{(n)},x^{(n-1)},y)|}{S_{|x|-n+1}}=\frac 1{S_{|x|-n+1}}\Bigg|\frac 1{m(x^{(n-1)})}-\frac 1{m(x^{(n)})}\Bigg| \gtrsim \frac 1{m(x^{(n-1)})S_{|x|-n+1}}=\frac 1{m\sigma(T_x^{(n-1)})}. \]
    \end{proof}

\noindent\textbf{Bekollé-Bonami weights on trees.} Let $\nu \in \mathcal C_\Q$. We aim to determine whether there are inclusion relations between the $B_p$ classes of weights associated with $\T$, $\U$ and $\Q$.
We show that in general, both the classes $B_p^\T$ and $B_p^\U$ are necessary in the intersection $B_p^{\Q}=B_p^{\T}\cap B_p^{\U}$. 

\begin{prop}
    Let $\nu\in\mathcal C_\Q$ and $1\leq p<\infty$. The following equivalences hold true:
    \begin{align*}
[\omega]_{p,\U}\lesssim [\omega]_{p,\T},\, \omega\in  B_p^\T
&\iff \nu(T_x)\lesssim\nu(U_x),\,x\in\X;\\
[\omega]_{p,\T}\lesssim [\omega]_{p,\U},\, \omega\in B_p^\U
&\iff \nu(U_x^{(1)})\lesssim\nu(T_x),\, x\in\X.
\end{align*}
\end{prop}
\begin{proof}
    We prove the result only for $p>1$. The proof for $p=1$ is analogous. Assume that there exists $C>1$ such that $\nu(T_x)\leq C\nu(U_x)$ for $x\in \X$. Take $\omega\in B_p^\T$. We have that 
    \[\avg{\omega}{U_x}\avg{\omega^{-\frac 1{p-1}}}{U_x}^{p-1}\leq C^p\frac {\omega(T_x)}{\nu(T_x)}\Big(\frac {\omega^{-\frac 1{p-1}}(T_x)}{\nu(T_x)}\Big)^{p-1}=C^p \avg{\omega}{T_x}\avg{\omega^{-\frac 1{p-1}}}{T_x}^{p-1},\]
    so that $[\omega]_{p,\U}\leq C^p[\omega]_{p,\T}$. Similarly, if $\nu(U_x^{(1)})\leq C\nu(T_x)$, $x\in\X$, then $[\omega]_{p,\T}\leq C^p[\omega]_{p,\U}$. Assume now $[\omega]_{p,\U}\leq D[\omega]_{p,\T}$, for some $D>0$. Let $x\in\X$. Since $q(x)\geq 2$, there exists $y\in s(x)$ such that $2\nu(T_y)\leq \nu(U_x)$. For $t>1$, we consider the weight $\omega_t=1+(t-1)1_{T_y}$. Notice that 
    \begin{align*}
        \avg{\omega_t}{U_x}\avg{\omega_t^{-\frac 1{p-1}}}{U_x}^{p-1}&=\Bigg(1+(t-1)\frac{\nu(T_y)}{\nu(U_x)}\Bigg) \Bigg(1+(t^{-\frac{1}{p-1}}-1)\frac{\nu(T_y)}{\nu(U_x)}\Bigg)^{p-1}\geq\frac t{2^{p-1}}\frac{\nu(T_y)}{\nu(U_x)}.
    \end{align*}
   On the other hand, since $\omega_t$ is constant on every sector that does not properly contain $T_y$, we have
    \begin{align*}
        [\omega_t]_{p,\T}&=\sup_{k\geq 1}\avg{\omega_t}{T_y^{(k)}}\avg{\omega_t^{-\frac 1{p-1}}}{T_y^{(k)}}^{p-1}\leq \sup_{k\geq 1} \Bigg(1+t\frac{\nu(T_y)}{\nu(T_y^{(k)})}\Bigg)\leq  1+t\frac{\nu(T_y)}{\nu(T_x)},
    \end{align*}
    by using that $\omega_t^{-\frac 1{p-1}}\leq 1$. By putting everything together, we have
    \[\frac t{2^{p-1}}\frac{\nu(T_y)}{\nu(U_x)} \leq [\omega_t]_{p,\U} \leq D[\omega_t]_{p,\T}\leq D\Bigg( 1+t\frac{\nu(T_y)}{\nu(T_x)}\Bigg).\]
    Dividing by $t\nu(T_y)$ both members and letting $t\to \infty$, we get  $\nu(T_x)\leq 2^{p-1}D\nu(U_x)$. Assume now $[\omega]_{p,\T}\leq D[\omega]_{p,\U}$, for some $D>0$. For $x\in\X$ and $t>1$ we consider the weight
    \[ \omega_t=1+(t-1)\delta_x.\]
    We argue as before, getting
    \[t\frac{\nu(x)}{\nu(T_x)}\frac{\nu(U_x)^{p-1}}{\nu(T_x)^{p-1}}\leq [\omega_t]_{p,\T}\leq D[\omega_t]_{p,\U}\leq D\Bigg(1+t\frac{\nu(x)}{\nu(U_{x^{(1)}})}\Bigg),\]
    and by letting $t\to\infty$ we have
    \[\frac{\nu(U_x)}{\nu(T_x)}\leq \Bigg(D\frac{\nu(T_x)}{\nu(U_{x^{(1)}})}\Bigg)^{\frac 1{p-1}}.\]
    We consider now the weight
    \[ \sigma_t=1+(t-1)1_{U_x}.\]
    When $0<t<1$, we have that
    \[ [\sigma_t]_{p,\T}\geq
\avg{\sigma_t}{T_x}\avg{\sigma_t^{-\frac 1{p-1}}}{T_x}^{p-1}\geq \frac{\nu(x)}{\nu(T_x)}\frac 1 t\Bigg(\frac{\nu(U_x)}{\nu(T_x)}\Bigg)^{p-1},\]
    and by Jensen's inequality
    \[[\sigma_t]_{p,\U}=\sup_{k\geq 1}\avg{\sigma_t}{U_x^{(k)}}\avg{\sigma_t^{-\frac 1{p-1}}}{U_x^{(k)}}^{p-1}\lesssim_p \sup_{k\geq 1}\Bigg(1+\frac 1 t\Bigg(\frac{\nu(U_x)}{\nu(U_x^{(k)})}\Bigg)^{p-1}\Bigg) \leq 1+\frac 1 t\Bigg(\frac{\nu(U_x)}{\nu(U_x^{(1)})}\Bigg)^{p-1}.\]
    We put everything together, by multiplying by $ t\nu(U_x)^{1-p}$ and letting $ t\to 0$, we have 
    \[\frac{\nu(x)}{\nu(T_x)}\leq \tilde D\Bigg(\frac{\nu(T_x)}{\nu(U_x^{(1)})}\Bigg)^{p-1},\]
    for some $\tilde D>0$ depending on $p$. Observe that for some $D'>0$ we have
    \[ 1=\frac{\nu(U_x)}{\nu(T_x)}+\frac{\nu(x)}{\nu(T_x)}\leq \Phi\Bigg(D'\frac{\nu(T_x)}{\nu(U_x^{(1)})}\Bigg),\qquad\text{with}\qquad\Phi(s)=s^{\frac{1}{p-1}}+s^{p-1}.\]
    By the fact that $\Phi$ is increasing and that $\Phi(s)\to 0$ as $s\to 0$, we get $\nu(U_x^{(1)})\lesssim\nu(T_x)$.
\end{proof}

It is possible to see that on a tree $\X$ with a measure $\nu\in\mathcal C_\Q$ there are in general no inclusions between the classes $B_p^\T$ and $B_p^\U$, even when allowing the geometry of the tree to be horocyclic.
\begin{example}
    Let $\X$ be a horocyclic tree with $q(n)\to\infty$ as $n\to\infty$ endowed with the horocyclic measure $\mu_\alpha$, $\alpha>1$, introduced in~\eqref{eq: def mu}.
Using that $q\geq2$, it is possible to
consider a family $\left\{ R_{n}\right\} _{n=0}^{\infty}\subseteq\mathcal{T}$
of sectors such that $R_n=T_{x_n}$ with $|x_n|=n+1$ and $R_{n+1}^{(2)}=R_n^{(1)}\neq R_{n-1}$. Observe that the members of the family are pairwise disjoint by construction.
    \begin{enumerate}
        \item We consider the weight
\begin{equation*}
\omega=1+\sum_{n=1}^{\infty}(q(n)^\alpha-1) 1_{R_{n}}.
\end{equation*}
    By direct computation we have 
    \[\avg{\omega}{\dot R_n^{(1)}}\avg{\omega^{-\frac 1{p-1}}}{\dot R_n^{(1)}}^{p-1}\geq\Big(1-\frac 1{q(n)}+q(n)^{\alpha-1}\Big)\avg{\omega^{-\frac 1{p-1}}}{\dot R_n^{(1)}}^{p-1}\gtrsim q(n)^{\alpha-1}\to\infty,\] so that $\omega\not\in B_p^\U$. Furthermore, $\omega\in B_p^\T$ since the nontrivial means on sectors are controlled by $\avg{\omega^{-\frac 1{p-1}}}{R_n^{(1)}}\leq 1$ and by \eqref{eq: mu set}
    \[\avg{\omega}{R_n^{(1)}}\leq 1+\sum_{m=n}^{\infty}\frac{\mu_\alpha(R_m)q(m)^{\alpha}}{\mu_\alpha(R_n^{(1)})}\lesssim 1+\sum_{m=n}^{\infty}\prod_{j=n}^{m-1}q(j)^{-\alpha} 
    \lesssim \sum_{m=n}^{\infty}2^{\alpha(n-m)}\lesssim 1.\]
    \item We consider the weight
    \[\omega=1+\sum_{n=1}^{\infty}(q(n)q(n+1)^{\alpha-1}-1)1_{\dot R_n}.\]
    Again, by direct computation, we can show that 
    \[\avg \omega{R_n}\avg{\omega^{-\frac 1{p-1}}}{R_n}^{p-1}\geq\frac{\omega(\dot R_n)}{\mu_\alpha(R_n)}\avg{\omega^{-\frac 1{p-1}}}{R_n}^{p-1}=q(n)q(n+1)^{\alpha-1}\frac{\mu_\alpha(\dot R_n)}{\mu_\alpha(R_n)}\avg{\omega^{-\frac 1{p-1}}}{R_n}^{p-1}\gtrsim q(n),\] 
    and so $\omega\not\in B_p^\T$. On the other hand, the weight is built in a way that $\omega\in B_p^\U$. Indeed, all the $U\in\U$ on which $\omega$ is not constant are the predecessors of $R_n$'s, where $\omega\geq 1$ and, since $q(m)q(m+1)^{\alpha-1}\mu_\alpha(\dot R_m)\approx \mu_\alpha(\dot R_{m-1})$, we get
    \[\avg{\omega}{\dot R_n^{(1)}}\leq 1+\sum_{m=n}^{\infty}\frac{\omega(\dot R_m)}{\mu_\alpha(\dot R_n^{(1)})}=1+\sum_{m=n}^{\infty}\frac{\mu_\alpha(\dot R_{m-1})}{\mu_\alpha(\dot R_{n-1})}\lesssim 1,\]
    as in the previous example.
    \end{enumerate}
\end{example}

The situation is different when $p=\infty$. Let $\nu\in\mathcal C_\Q$.  From Proposition~\ref{prop: infCarl} and the fact that $\mathcal C_\T=\mathcal C_\U$, we have that $B_\infty^\T=B_\infty^\U$.

In the context of trees, we also investigate whether there is any connection with the weight class $\BBv(\nu)$ introduced in \cite{CADMMRV2025} and defined by

\begin{equation}\label{BBv}
    \BBv(\nu):=\Bigl\{\omega\colon\X\to (0,\infty)\colon [\omega]_{\BBv(\nu)}:=\sup_{\substack{Q,R\in\T\\Q^{(1)}=R^{(1)}}}\avg{\omega}{Q}\avg{\omega^{-\frac{p'}{p}}}{R}^{\frac{p}{p'}}<\infty\Bigr\}.
\end{equation}

In what follows, we show that $B_p^\Q$ strictly contains $\BBv(\nu)$. In particular, this shows that the weighted inequalities for $\Proj$ established in~\cite{CADMMRV2025} were obtained under a non-sharp weight condition. On the other hand, $B_p^\Q$ is the largest possible weight class for $\Proj^+$ and, consequently, provides a natural sufficient condition for the boundedness of $\Proj$.

\begin{theorem}
    For every $p\in (1,\infty)$, we have $\BBv\subsetneq B_p^\Q$ whenever $q$ is unbounded.
\end{theorem}
        \begin{proof}
Let $\omega\in\BBv$. Taking $Q=R$ in the definition of $[\omega]_{\BBv}$, we immediately have $[\omega]_{p,\T}\leq[\omega]_{\BBv}$.
It remains to control the characteristic associated with $\U$. Fix $x\in\X$. We have
\[\avg \omega{U_x}=\sum_{y\in s(x)}\avg \omega{T_y}\frac{\nu(T_y)}{\nu(U_x)}\leq \max_{y\in s(x)}\avg\omega{T_y}.\]
Similarly, 
\[\avg{\omega^{-\frac 1{p-1}}}{U_x}\leq \max_{z\in s(x)}\avg{\omega^{-\frac 1{p-1}}}{T_z},\]
so that
\[\avg \omega{U_x}\avg{\omega^{-\frac 1{p-1}}}{U_x}^{p-1}\leq \max_{y,z\in s(x)}\avg\omega{T_y} \avg{\omega^{-\frac 1{p-1}}}{T_z}^{p-1}\leq [\omega]_{\BBv}.\]

We now build an example to show that the containment is proper.
Consider a horocyclic tree such that $\|q\|_\infty=\infty$ and assume $\nu\in \mathcal C_\Q$ to be horocyclic. Let $\{ Q_{n}\} _{n\in\N},\{ R_{n}\}_{n=1}^{\infty}\subseteq\mathcal{T}$ such that $Q_{n+1}^{(2)}=R_{n+1}^{(2)}$ and $Q_{n+1}^{(1)}\neq R_{n+1}^{(1)}=Q_n$. Observe that the family $\{ R_{n}\}_{n=1}^{\infty}$ is pairwise disjoint by construction. Let $x_n\in \X$ be such that $Q_n=T_{x_n}$. We choose the family in such a way that $\sup_{n\in\N}q(x_n)=+\infty$. We consider the weight
\[\omega =1+\sum_{n\in\mathbb N}(q(x_n)^{\frac 12}-1)1_{R_{n+1}}.\]
Observe that $\omega\not\in \BBv$. Indeed,
\[[\omega]_{\BBv}\geq
\sup_{n\in\N}\Bigg[\avg\omega{R_{n+1}}\sup_{\substack{Q\in\T\\Q^{(1)}=Q_n}}\avg{\omega^{-\frac 1{p-1}}}Q\Bigg]\gtrsim \sup_{n\in\N}\avg \omega{R_{n+1}}=\sup_{n\in\N} q(x_n)^{\frac 12}=+\infty. \]
On the other hand, by the fact that $\nu$ is horocyclic, we have $q(x_n)\nu(R_{n+1})\leq\nu(Q_n)$. Hence, $\omega(R_{n+1})=q(x_n)^{\frac 12}\nu(R_{n+1})\leq \nu(Q_n)$ so that we can easily conclude that $\avg \omega{Q_n}\lesssim 1$. Furthermore, for $m\geq n$, $\omega(Q_m)\leq q(x_{m-1})^{\frac 12}\nu(R_m)\leq q(x_{m-1})^{-\frac 12} \nu(Q_{m-1})$ and, by iteration, we have
\[\avg \omega{Q_n^{(1)}}\leq  1+\frac{1}{\nu(Q_n^{(1)})}\sum_{m=n}^{\infty}\frac{\nu(Q_{m-1})}{q(x_{m-1})^{\frac 12}}\leq 1+\sum_{m=n}^{\infty}\frac{1}{q(x_{m-1})^{\frac 12}}\prod_{j=n}^{m-1} \frac 1{q(x_{j-1})}\lesssim 1,\]
by the fact that $q\geq 2$.
Using that $\omega\geq 1$ everywhere and that it is constant on all the other sectors, we conclude that $\omega\in B_p^\T$. Similarly, one proves that $\omega\in B_p^\U$.
\end{proof}

\begin{rem}
    Let $\nu\in\mathcal C_\Q$ be a horocyclic measure on a horocyclic tree $\X$, that is $\nu(x)=\nu(y)$ and $q(x)=q(y)$ if $|x|=|y|$. Let $\omega\colon \mathfrak{X}\to (0,+\infty)$ be a horocyclic weight. We have that 
    \[[\omega]_{p,\T}=[\omega]_{p,\U}=[\omega]_{p,\Q},\qquad 1\leq p<\infty. \]
    Furthermore, when $p>1$, we also have $[\omega]_{p,\Q}=[\omega]_{\BBv}$.
    
    Indeed, for the case $p>1$, observe that whenever $Q,R\in\T$ are such that $Q^{(1)}=R^{(1)}$, we have that $\omega(Q)=\omega(R)$ and similarly for $\omega^{-\frac{1}{p-1}}$, hence $\avg{\omega}{Q}=\avg{\omega}{R}$ and $\avg{\omega^{-\frac{1}{p-1}}}{Q}=\avg{\omega^{-\frac{1}{p-1}}}{R}$. Then
        \begin{equation*}
            [\omega]_{\BBv}= \sup_{\substack{Q,R\in\T\\Q^{(1)}=R^{(1)}}}\avg{\omega}{Q}\avg{\omega^{-\frac{1}{p-1}}}{R}^{p-1}=\sup_{Q\in\T}\avg{\omega}{Q}\avg{\omega^{-\frac{1}{p-1}}}{Q}^{p-1}=[\omega]_{p,\T}.
        \end{equation*}
        It is now sufficient to see that $[\omega]_{p,\T}=[\omega]_{p,\U}$. Notice that for a horocyclic weight $\eta\colon \mathfrak{X}\to (0,+\infty)$, we have
            \begin{align}\label{eq: media di eta}
                \avg{\eta}{\anc{x_Q}{1}}=\frac{1}{q(\anc{x}{1}_Q)\nu(Q)}\eta(\anc{x_Q}{1})=\frac{1}{q(\anc{x}{1}_Q)}\sum_{\substack{R\in\T\\R^{(1)}=Q^{(1)}}}\frac{\eta(R)}{\nu(Q)}=\avg{\eta}{Q}.
        \end{align}
        Hence we get
        \begin{align*}
            [\omega]_{p,\T}= \sup_{Q\in\T}\avg{\omega}{Q}\avg{\omega^{-\frac 1{p-1}}}{Q}^{p-1}
            =\sup_{Q\in\T}\avg{\omega}{\anc{\dot Q}{1}}\avg{\omega^{-\frac 1{p-1}}}{\anc{\dot Q}{1}}^{p-1}=\sup_{U\in\U}\avg{\omega}{U}\avg{\omega^{-\frac 1{p-1}}}{U}^{p-1}=[\omega]_{p,\U}.
        \end{align*}
        Finally, for the case $p=1$, we use \eqref{eq: media di eta} together with the fact that, for every $T\in\T$,
\begin{equation*}\|\omega^{-1}1_{\anc{\dot T}{1}}\|_{\infty}=\max_{\substack {R\in\T\\R^{(1)}=T^{(1)}}}\|\omega^{-1}1_{R}\|_{\infty}=\|\omega^{-1}1_{T}\|_{\infty}, \end{equation*}
        to say that
        \begin{equation*}[\omega]_{1,\T}
        =\sup_{Q\in\T}\avg{\omega}{Q}\|\omega^{-1}1_{Q}\|_\infty
        =\sup_{Q\in\T}\avg{\omega}{\anc{\dot Q}{1}}\|\omega^{-1}1_{\anc{\dot Q}{1}}\|_\infty
        =\sup_{U\in\U}\avg{\omega}{U}\|\omega^{-1}1_{U}\|_\infty
        =[\omega]_{1,\U}.
        \end{equation*}
\end{rem}

\noindent\textbf{Acknowledgments.} J.M. Conde-Alonso was partially supported by the 2025 Leonardo Grant for Scientific Research and Cultural Creation \texttt{LEO25-1-19094-CBB-MAT-169} (BBVA Foundation). The BBVA Foundation accepts no responsibility for the opinions, statements and contents included in the project and/or the results thereof, which are entirely the responsibility of the authors. M.~Monti and E.~Rizzo are members of GNAMPA–INdAM. M.~Monti is partially supported by the GNAMPA–INdAM Project 2026 \lq\lq Transferring Harmonic Analysis between Discrete Structures and Manifolds" (CUP E53C25002010001). E.~Rizzo was partially supported by the GNAMPA–INdAM Project 2025 \lq\lq Function Spaces and Applications" (CUP E5324001950001). A large portion of this work was completed while E.~Rizzo was a post-Doc at the Università degli Studi di Milano, +Dip Eccellenza.

\bibliographystyle{abbrv}
\bibliography{biblio}

\end{document}